\documentclass[%draft,a4paper,
10pt]{amsart}
\usepackage{graphicx}
\usepackage{psfrag}

\newtheorem{theorem}{Theorem}[section]
\newtheorem{lemma}[theorem]{Lemma}

\newtheorem*{conjecture*}{Conjecture}
\newtheorem*{theorema}{Theorem A}
\newtheorem*{theoremb}{Theorem B}
\newtheorem*{theoremc}{Theorem C}

\newtheorem{proposition}[theorem]{Proposition}
\theoremstyle{definition}

\theoremstyle{question}

\newtheorem*{question*}{Open Questions}
\theoremstyle{acknowledgement}

\newtheorem*{theorem*}{Theorem}
\numberwithin{equation}{section}

\newcommand{\Er}{\mathbb{R}}

\newcommand{\calh}{{\mathcal{H}}}

\newcommand{\call}{{\mathcal{L}}}

\newcommand{\be}{\mathbf{e}}
\newcommand{\bE}{\mathbf{E}}

\newcommand{\p}{\phi(l_1, l_2, x)}
\newcommand{\pr}{\phi_{l_i} }
\newcommand{\ps}{\phi_{l_i l_j} }
\newcommand{\pt}{\phi_{l_i l_j l_k } }

\newcommand{\pu}{\phi_{l_1}}
\newcommand{\pv}{\phi_{l_1 l_1}}
\newcommand{\pw}{\phi_{l_1 l_2} }
\newcommand{\pvv}{\phi_{l_1l_1l_1} }
\newcommand{\pvw}{\phi_{l_1l_1l_2} }
\newcommand{\pww}{\phi_{l_1l_2l_2} }

\def\pmb#1{\setbox0=\hbox{#1}%
\kern-.02em\copy0\kern-\wd0
\kern0.01em\raise.01732em\copy0\kern-\wd0
\kern0.01em\raise.01732em\copy0\kern-\wd0
\lower.01732em\copy0\kern-\wd0 \kern0.01em\copy0\kern-\wd0
\kern0.01em\lower.01732em\box0}

\begin{document}

\title[Symmetry of Traveling Wave Solutionss]
{Symmetry of Traveling Wave Solutions to the Allen-Cahn Equation in $\Er^2$ }

\author{Changfeng Gui}
\address{Changfeng Gui,  School of Mathematics, Department of Mathematics, U-9, University of Connecticut\\
Storrs, CT 06269, USA} \email{gui@math.uconn.edu}
%%Put your correct addresss

\maketitle

%\date{May 2, 2007}
%\today

\begin{abstract}
In this paper, we prove even symmetry of  monotone traveling wave
solutions to the balanced Allen-Cahn equation in the entire plane.
Related results for the unbalanced Allen-Cahn equation are also
discussed.
\end{abstract}


\section{Introduction}

The study of traveling wave solutions  is a classic area  of
research of reaction diffusion equations. In  the last decade,
traveling wave solutions and generalized traveling wave solutions
have generated a lot of  excitements among mathematicians, partially
due to rich   phenomena  in various branches of applied sciences
which are related to  traveling  fronts, such as flame propagation
in various media, population spreading, etc; The  research is also
fueled by  new discoveries of  deep and beautiful mathematics
related to traveling waves. See, for example, a recent survey
\cite{bh1} and a monograph  \cite{bh2} for details. In this paper,
we are mainly concerned with traveling wave solutions in the entire
plane of the Allen-Cahn equation with a balanced double well
potential, even though we also discuss Allen-Cahn equation with an
unbalanced potential or in the entire higher dimensional space.
Namely, we consider a traveling wave solution $v(x, y, t)=u(x,
y-ct)$ of the Allen-Cahn equation
\begin{equation}\label{allen-cahn}
v_t=\Delta_{x} v+ v_{yy}-F'(v), \quad  ( x, y, t) \in
\Er^{n-1}\times \Er \times\Er^+
\end{equation}
where $c>0$ and  $F$ is a double-well potential, i.e., $F$ is $C^{3}
$ and satisfies
\begin{equation}\label{doublewell}
\left\{
\begin{split}
&  F'(-1)=F'(1)=0, \quad  F''(-1)>0, F''(1)>0 \\
&F'(s)>0, \, s \in (-1, \theta); \quad F'(s)<0,  \, s \in (\theta,
1)
\end{split}
\right.
\end{equation}
for some $\theta \in (0, 1)$.  Without loss of generality, we may
assume that $F(-1)=0$ and $\theta=0$.  If $ F(1)=F(-1)=0$,  $F$ is
called a balanced double well potential. Otherwise, it is called an
unbalanced double well potential,  and in this case we may assume
that $F(1)>F(-1)=0$ without loss of generality.

A typical example of balanced  double well potential is  $
 F(u)=\frac{1}{4} (1-u^2)^2, \quad u \in \Er
 $, while a typical  unbalanced double well potential is
 $F(u)=\frac{1}{4} (1-u^2)^2-  a ( u^3/3-u)$  with $ a \in (-1, 0)$.
 Note that $F'(u)=(u-a)(u^2-1)$ in the latter case.

The value of  $u(x, y)$  may be restricted to $[-1, 1]$.  It is
obvious that  $u$ satisfies an elliptic equation
\begin{equation}\label{eq-wave}
\Delta_{x} u+ u_{yy}+ c u_y-F'(u)=0, \quad |u| \le 1, \quad (x, y)
\in \Er^{n-1} \times \Er.
\end{equation}

We may  assume  that the traveling  wave solution  is monotone in
time and  hence in the direction of $y$. Without loss of generality,
we assume

\begin{equation}\label{monotone}
u_y(x, y) >0, \quad (x, y) \in \Er^{n}
\end{equation}

We may also assume that the solution $u$ connects two stable states,
i.e.,
\begin{equation}\label{limit-tw}
\lim_{y \to \pm \infty} u(x, y) =\pm 1, \quad  x \in \Er^{n-1}.
\end{equation}

We note that the limit condition above does not  need to  be uniform
in $x$. Indeed, we shall see that the limits are not uniform.  When
$n=1$  there exists a unique speed $c_0 \ge 0$ such that
\eqref{eq-wave}   has a unique solution $g(y)$ (up to translation)
satisfying the monotone condition \eqref{monotone}, i.e.,
\begin{equation}\label{1d-ac}
\left \{
\begin{split}
&g''(s)+c_0 g'(s)-F'(g(s))=0, \quad s \in \Er,\\
& \lim_{s \to \infty} g(s)=1,\quad \lim_{s \to -\infty} g(s)=-1.
\end{split}
\right.
\end{equation}
where  $c_0=0$  in the balanced case and  $c_0>0$ in the unbalanced
case. We may assume that $g(0)=0$.  The solution $g$ is
non-degenerate in the sense that the linearized operator has a
kernel spanned only by $g'$.

It is well-known that  when $F$ is  balanced, $g$ is a minimizer of
the following energy functional
$$
\bE(v):=\int_{-\infty}^{\infty} [\frac{1}{2} |v'|^2 + F(v)]dx
$$
in $\calh:=\{ v \in H^1_{loc}(\Er): -1 \le v \le 1, \,\,\lim_{ s \to
\pm \infty} v(s)=\pm 1\}$ and
$$
\be:=\bE(g)=\int_{-1}^{1} \sqrt{ 2 F(u)} du<\infty.
$$

There is a significant difference between the balanced and the
unbalanced Allen-Cahn equation  when  traveling wave solutions are
concerned. The difference of zero speed and positive speed of one
dimensional traveling wave solution  $g$  for the balanced and
unbalanced potential leads to a fundamental difference of the
structure of traveling fronts in higher dimensional spaces, as
discussed below, as well as shown in Theorem 1.1 and Theorem 3.2 and
\cite{gui-hd}.
 The existence, uniqueness, stability  and other
qualitative properties of traveling wave solutions to the unbalanced
Allen-Cahn equation have been  studied in
 \cite{hmr2} \cite{hmr3}, \cite{nt1}, \cite{nt2}, \cite{tani1},
 \cite{tani2}. Similar traveling wave solutions for Fisher-KPP type
 equation or combustion equation have  also been investigated  in \cite{bonn},
 \cite{hm1}, \cite{hn}, \cite{mn}.    The typical shape of traveling
 fronts studied in these articles are conical.  The stability and uniqueness results
 are also based on the assumption that the traveling fronts are
 conical. In particular, the traveling fronts for these equations are  globally  Lipschitz
 continuous.
 Traveling wave solutions for the balanced Allen-Cahn equation are first studied in   \cite{chen},
 where non-conical and non-planar traveling fronts with axial
 symmetry are proven existing.  It is noted that the traveling fronts are not globally Lipschitz.
  Indeed, the following theorem is
 proven in \cite{chen}.

\begin{theorema}[Chen, Guo, Hamel, Ninomiya, Roquejoffre, 2007]
 For any $c>0$, there exists a solution $U(x, y)
=U(|x|, y) $ to \eqref{eq-wave},  \eqref{monotone}, \eqref{limit-tw}
such that $U_r(r, y)<0$ for $r>0$ and $U(0,0)=0$.  Furthermore, if
the 0-level set of $U$ is denoted by $\Gamma $, then

(i) when $n=2$, $\Gamma$ is asymptotically a hyperbolic cosine
curve, i.e., for some $A>0$

\begin{equation}\label{asym-2}
\lim_{y \to \infty, U(x, y)=0} \frac{ cosh(2 \mu x)}{\mu y}=
\frac{A}{c}
\end{equation}
where $\mu=\sqrt{F''(1)}$.

(ii) when $n>2$, $\Gamma$ is asymptotically a paraboloid, i.e.,
\begin{equation}\label{asym-n}
\lim_{y \to \infty, U(x, y)=0} \frac{|x|^2}{2y} =\frac{n-2}{c}.
\end{equation}
\end{theorema}

It is very interesting to note that  for $n>2$  the traveling fronts
are very similar to the translating radial solutions to the mean
curvature flow, i.e., the entire radial solutions to
\begin{equation} \label{meancurvature}
div ( \frac{D\Gamma}{\sqrt{1+|D\Gamma|^2}})=\frac{1}{\sqrt{1+|D
\Gamma|^2}}\ \ in \ \Er^{n-1},
\end{equation}
where $y=\Gamma(|x|)$ can be computed as
\begin{equation}\label{soliton}
\Gamma(r)=\frac{r^2}{2(n-2)}-\ln r +C_1 -
\frac{(n-2)(n-5)}{2}r^{-2}+o(r^{-2}).
\end{equation}
See \cite{aag}, \cite{gjj}. This is not surprising due to the
connection between the surface motion by mean curvature and the
interface motion of solutions to the balanced  Allen-Cahn equation.
See, for example,  \cite{chen2}, \cite{schat}, \cite{ess}, \cite{ilmanen}.  It is reasonable to
expect that the traveling fronts with  unit speed  $c=1$ should be
related to the translating mean curvature flow  of unit speed.  The
case $n=2$ is slightly different, in this case the solution for the
translating mean curvature flow is the `` grim reaper '', i.e., the
curve given by $\Gamma(x)=\log{sec(x)}$, while the traveling front
is a hyperbolic cosine.  The discrepancy between  these two curves
is due to the strong interaction caused by  the reaction term in the
Allen-Cahn equation.

Recent studies  on the translating  mean curvature flow reveal very
interesting properties of convex solutions.  See \cite{white},
\cite{wang}, \cite{gjj} and references therein.  In particular, it
is shown in \cite{wang} that convex solutions to
\eqref{meancurvature} must be rotationally symmetric for $n\le 3$.
It is then natural to ask whether a traveling wave solution to
\eqref{allen-cahn} with monotone  \eqref{monotone} and limit
condition \eqref{limit-tw} must be rotationally symmetric, or, in
the terminology of this paper, axially symmetric after a proper
translation in $x$ variable.  In this paper, we shall show that this
is indeed true for $n=2$.  To be more precise,  we have the
following main theorem.

\begin{theorem} Assume that $F$ is a balanced double well potential satisfying
\eqref{doublewell} and $F(-1)=F(1)=0$.  Suppose $u$ satisfies
\eqref{eq-wave}, \eqref{monotone} and \eqref{limit-tw}. Then, when
$n=2$, $u$ is evenly symmetric with respect to $x$  after a proper
translation, and $ u_x(x, y) <0$ for $ x>0$.
\end{theorem}

  In  dimensions $n \ge 3$, obviously we need  more conditions,
since if $u(x, y) $ is a solution in $\Er^n$,  then a trivial
extension $u(x,s, y)=u(x, y)$ is a solution in $\Er^{n+1}$.  It
remains open  whether all monotone traveling wave solutions with the
limit condition \eqref{limit-tw} must be either  axially symmetric
or trivial extension of a axially symmetric solution in lower
dimensional space.  Due to possible existence of non-rotationally
convex translating mean curvature flow in higher dimensions
(\cite{white}, \cite{wang}), the answer for the  above question is
probably not affirmative except for $n=3$.  The latter  will be
discussed in  a forthcoming paper  \cite{gui-hd}. We note that
 symmetry results have  also  been proven  for certain saddle
solutions of Allen-Cahn equation \eqref{allencahn-general} in
\cite{gui-saddle} and for solutions of nonlinear stationary
Schrodinger equation in \cite{gui-schro}.

A very closely related question is the De Giorgi conjecture, which
may be regarded as assertion on the one dimensional symmetry  of
solutions to \eqref{eq-wave} when $c=0$, i.e.,
\begin{equation}\label{allencahn-general}
\Delta u-F'(u)=0, \quad |u|<1, \quad (x, y)  \in \Er^n.
\end{equation}

The conjecture may be stated as follows.

\begin{conjecture*}(De Giorgi, 78)
If $u$ satisfies \eqref{allencahn-general} and \eqref{monotone},
then for at least $n \le 8$,  $u$ must be a one dimensional
solution, i.e. a proper trivial  extension, rotation  and
translation of $g$. In other words, the level sets of $u$ must be
hyper planes.
\end{conjecture*}

This  conjecture is based on the famous Bernstein problem regarding
the classification of  complete minimal graph in $\Er^n$
(\cite{bdg}, \cite{giu}). The  De Giorgi conjecture is proven
affirmatively for $n=2$ in \cite{gg1} and for $n=3$ in \cite{ac}.
With the extra limit condition \eqref{limit-tw}, it is proven for $n
\le 8$ in \cite{savin}. Recently, non planar solutions for
\eqref{allencahn-general}  with $n \ge 9$ are constructed in
\cite{dkw2}  by using the non-planar minimal graph by Bombieri, De
Giorgi and Giusti (\cite{bdg}, \cite{giu}).

For the  case of an unbalanced double well potential, we shall show
a similar result as Theorem 1.1, which  improves  a classification
theorem of \eqref{hmr3} for all monotone traveling wave solutions in
$\Er^2$. See Theorem 3.2. in Section 3.

 The paper is organized as
follows.  In Section 2,  the main result Theorem 1.1 shall be
proved. Theorem 3.2., the classification result  for traveling wave
solutions of the unbalanced Allen-Cahn equation in $\Er^2$,  will be
proved in Section 3. Finally, traveling waves solutions connecting
various stationary one dimensional solutions will be investigated in
Section 4.

\section{Even Symmetry of Traveling Wave Solutions of the Balanced Allen-Cahn Equation in $\Er^2$}

Through out in this section,  we assume that  $n=2$ and  the double
well potential $F$ is balanced, i.e., $F(-1)=F(1)=0$.  We shall
prove Theorem 1.1 in three main steps.  First we carry out a
preliminary asymptotical analysis of the level sets of the solution
$u$ and show that the  slope of the  $0$-level curve $y=\gamma(x)$
must tend to $\pm \infty $ as $x$ tends to $\pm \infty.$  Second, we
show that $y=\gamma(x)$ is asymptotically  hyperbolic cosine and
obtain a very detailed asymptotical formula.  Last, we complete the
proof by  using the asymptotical formula  of the level curve and the
moving plane method.  We note that the regularity condition of $F$
can be replaced by $C^{2, \beta}$ with some $\beta \in (0, 1)$ for
most discussion below, except in \eqref{relation2} where the third
derivatives of $\phi$ with respect to $l_i$ require $F \in C^{3}$.

\subsection{Preliminary Analysis of the level set}

\vskip0.2in

We first show an important lemma which asserts the integrability of
of $u_y$.

\begin{lemma}
Suppose that $u$ is a solution to \eqref{eq-wave}, \eqref{monotone}
and \eqref{limit-tw}.  Then
\begin{equation}\label{u_y}
\int_{\Er^2} u_y^2 dxdy <\infty.
\end{equation}
\end{lemma}

\begin{proof}

Define
$$
h(x)=\int_{\Er} u_x u_y dy, \quad x \in \Er.
$$
Since $u$ is bounded in $C^3(\Er^n)$ by the standard elliptic
estimates and $u_y$ is positive,  it is easy to see that $h(x)$ is
well-defined and
$$
|h(x)|< C, \quad x \in \Er
$$
for some constant $C>0$.

Note that due to \eqref{limit-tw}, we have
$$ \lim_{y \to \pm \infty} u_x=0, \quad \lim_{y \to \pm \infty} u_y=0,
\quad x \in \Er.
$$

Differentiating $h(x)$ with respect to $x$ and using the equation,
we obtain
\begin{equation}\label{u_y1}
\begin{split}
h'(x)=&\int_{\Er} ( u_{xx}u_y + u_x u_{xy}) dy \\
=&\int_{\Er}  [\frac{\partial}{\partial y} \bigl(F(u)-\frac{1}{2} u_{y}^2  +\frac{1}{2} u_x^2  \bigr)-c u_y^2 ]dy\\
=& -c\int_{\Er} u_y^2dy.
\end{split}
\end{equation}
Then
\begin{equation}\label{u_y2}
\int_{a}^{b} \int_{\Er} u_y^2dydx =\frac{1}{c} \bigl(h((a)-h(b)
\bigr).
\end{equation}

The bound of $h(x)$  immediately leads to the integrability of
$u_y^2$ in $\Er^2$.
\end{proof}

 Due to \eqref{monotone} and \eqref{limit-tw},  the 0-level set of $u$
is a $C^3$  graph of a function defined in $\Er$.   We let $
y=\gamma(x), x \in \Er $ be such a function.  The next lemma asserts
that the slope of $y=\gamma(x)$ must tend to infinity as $x $ goes
to infinity.

\begin{lemma}\label{slope1}  There holds
\begin{equation}\label{slope}
\lim_{ |x| \to \infty} | \gamma'(x)|=\infty.
\end{equation}
\end{lemma}

\begin{proof}
Since $u$ is bounded in $C^3(\Er^2)$,  Lemma 2.1 implies that
$$
\lim_{|x| \to \infty} u_y(x, y)=0, \quad \text{uniformly in } y \in
\Er.
$$

Now assume that \eqref{slope} is not true,  then there exists a
sequence $\{ x_m\} $ such that $ |x_m| $ goes to infinity and
$$
\lim_{m \to \infty} \gamma'(x_m) =k_0
$$
for some constant $k_0$.

We shall translate $u$ along this sequence of $x_m$.  Define
$$
u_m(x, y)= u( x+x_m, y+\gamma(x_m)), \quad  (x, y) \in \Er^2.
$$

By the standard theory for elliptic equations, we know that $u_m$ is
bounded in $C^{3, \beta} (\Er^2)$. Then there is a subsequence,
which we still denote by $\{x_m\}$, such that $u_m $ converges to a
function $u_*$ in $C^3_{loc} (\Er^2)$.  It is easy to see that
 $ u_*(0, 0)=0, \,\,  \frac{\partial} {\partial y} u_*(x, y) =0,\,
(x, y) \in \Er^2. $  Then $u_*(x, y)= g_*(x)$ for some $C^3$
function $g_*$ which  is a solution to the one dimensional
stationary Allen-Cahn equation
\begin{equation}\label{allencahn-1d}
u_{xx}-F'(u)=0, \quad x \in \Er.
\end{equation}
Furthermore, since $u(x, \gamma(x))=0$ and hence
$$ u_x( x,
\gamma(x))+ u_y(x, \gamma(x)) \gamma'(x)=0, \quad x \in \Er,
$$
we obtain
$$
g_*'(0)=\lim_{m \to \infty}  u_x(x_m, \gamma(x_m))= -\lim_{n \to
\infty} u_y(x_m, \gamma(x_m)) \gamma'(x_m)=0.
$$
Then we conclude that $g_* \equiv 0$.  We claim that this will lead
to a contradiction.

As in the proof of Lemma 2.1,  we define
$$
h_m(x)=\int_{-\infty}^0 \frac{\partial u_m}{\partial
x}\frac{\partial u_m}{\partial y } dy.
$$
It is easy to see that $|h_m(x)|<C$ for some constant independent of
both  $ x $ and $ m$. We can  also derive

\begin{equation*}
h'_m(x)=-c \int_{-\infty}^0 (\frac{\partial u_m}{\partial y})^2 dy +
\frac{1}{2} (\frac{\partial u_m}{\partial x})^2 (x, 0) -\frac{1}{2}
(\frac{\partial u_m}{\partial y})^2 (x, 0)+F( u_m(x, 0)), \quad x
\in \Er
\end{equation*}

For any fix $R>0$,  in view of \eqref{u_y} we have
\begin{equation*}
\int_{-R}^R \bigl[\frac{1}{2} (\frac{\partial u_m}{\partial x})^2
(x, 0) -\frac{1}{2} (\frac{\partial u_m}{\partial y})^2 (x, 0)+F(
u_m(x, 0))]\bigr] dx<C
\end{equation*}
for some constant $C$ independent of $m, R$.

Letting $m$ go to infinity, we obtain $ 2 F(0) R \le C $,  which is
a contradiction.  The proof  of the lemma is then  complete.

\end{proof}

Indeed, we conclude that  the level curve must be of one of the
following four possibilities:

\begin{equation*}
\begin{split}
& \hskip-1.5in \text{\bf (i) } \quad \quad  \lim_{ x \to \infty}
\gamma'(x) =+\infty, \quad\quad \lim_{ x \to -\infty} \gamma'(x) = -\infty; \\
& \hskip-1.5in \text{\bf (ii)} \quad \quad  \lim_{ x \to \infty}
\gamma'(x) =+\infty, \quad  \quad\lim_{ x \to -\infty} \gamma'(x) = +\infty;\\
&\hskip-1.5in \text{\bf (iii)} \quad \quad \lim_{ x \to \infty}
\gamma'(x) =-\infty, \quad \quad\lim_{ x \to -\infty} \gamma'(x) = -\infty; \\
& \hskip-1.5in \text{\bf (iv) } \quad \quad \lim_{ x \to \infty}
\gamma'(x) =-\infty, \quad\quad  \lim_{ x \to -\infty} \gamma'(x)
=+\infty.
\end{split}
\end{equation*}

Moreover, it can also be concluded  by the arguments above  that the
profile of $u$ along the level curve  must be approximately the one
dimensional transition layer  $g(x)$ or $g(-x)$. To be more precise,
 we define
$$ u_s(x, y): =u(s+x, \gamma(s)+y), \quad (x, y) \in \Er^2.
$$
The following  lemma holds.

\begin{lemma}\label{profile}
The translated solution $u_s(x, y)$ converges in $C^3_{loc} (\Er^2)$
to either $g(x)$ or $g(-x)$ as $|s|$ tends to infinity.
\end{lemma}

\subsection{ The exponential decay of $u$ and the  Hamiltonian identity}

\vskip0.2in

 In this subsection,  we shall show that  solution $u$ must decay exponentially to
$\pm 1$ as the distance from the the level set $y=\gamma(x)$ tends
to infinity.  The exponential decay of $u$ will be used to prove  a
version of Hamiltonian identity for equation \eqref{eq-wave}.   This
type of  analysis was first  carried out  in \cite{chen} for the
axially symmetric traveling wave solutions. Their arguments are
slightly modified and presented here for the convenience of the
reader.

 Due to the double well potential condition of $F$,
 there exist two constants $\alpha^+, \alpha^-$ such that
 $-1< \alpha^-<0< \alpha^+ <1$ and
 $$
 F''(s)>\mu_0>0, \quad  s \in [-1, \alpha^-]\cup [\alpha^+, 1].
 $$
 for some constant $\mu_0>0$.

Define
\begin{equation*}
\begin{split}
& \Omega^+:= \{ (x, y)  \in \Er^2:   u(x, y) \ge \alpha^+\},  \quad
\Omega^-:= \{ (x, y) \in \Er^2:   u(x, y) \le \alpha^-\},\\
& \Omega^0:= \{ (x, y)  \in \Er^2:   \alpha^- \le u(x, y) \le
\alpha^+\}, \quad \Omega^0_y:= \{ x  \in \Er:   \alpha^- \le u(x, y)
\le
\alpha^+\},\\
&\gamma^\alpha:=\{ (x, y) \in \Er^2: y=\gamma^\alpha(x), \,\, u(x,
\gamma^\alpha(x) )=\alpha \}, \quad  \alpha \in (-1, 1) .
\end{split}
\end{equation*}

By Lemmas \ref{slope1} and \ref{profile}, it is easy to see that $
meas (\Omega_y^0) < K <\infty$ for some constant $K$ independent of
$y$.  Indeed, there exists  a positive constant $Y_0>0$  and two
$C^3$  functions $ x=k_i(y),  i=1, 2 $    such that $ \gamma^0 \cap
\{ (x, y) \in \Er^2: |y|> Y_0\}$ can be expressed as the graph of
$k_i(y)$, i.e.,
$$
\gamma^0 \cap \{ (x, y) \in \Er^2: |y|> Y_0\}=\{ (x, y): x=k_i(y),\,
|y| > Y_0, \, i=1, 2\}.
$$

In Case (i),  both $k_1$ and $k_2$ are defined for $Y>Y_0$,  while
in Case (iv),  $k_1$ and $k_2$ are defined for $Y<-Y_0$. We may
assume that $k_1(y)<k_2(y)$ in these two cases.

 In Case (ii) and
(iii),  $k_1$ is defined for $y>Y_0$ and $k_2$ is defined for
$y<-Y_0$.

In all cases, we have
\begin{equation}\label{width}
|x-k_1(y)| <  K,  \,\, \text{or}  \, \, |x-k_2(y)| <  K, \,  \,
\quad
 \forall x \in \Omega_y^0, \,\, |y|> Y_0.
 \end{equation}

Now we can state the exponential decay of $u$ as the following
lemma.

\begin{lemma}\label{exponential1}
There exist  constants $C$ and $\nu>0$ such that
\begin{equation}\label{exponential}
\left\{
\begin{split}
&|u^2-1|+|\nabla u|+|\nabla^2 u| \le C e^{ -\nu d(x, y)} , \quad  |y|>Y_0\\
&|u^2-1|+ |\nabla u|+|\nabla^2 u| \le C e^{-\nu |x| }, \quad |y|\le
Y_0.
\end{split}
\right.
\end{equation}
where $d(x, y):=\min\{ |x-k_1(y)|, \,\, |x-k_2(y)|\}$ for $
|y|>Y_0$.

\end{lemma}

\begin{proof}

Let
$$ w(x, y)= 1\mp u(x, y)>0,  \,\,  (x, y) \in \Omega^{\pm}.
$$
Then, by the definition of $\mu_0$ and $\Omega^{\pm}$, it is easy to
see that
$$
w_{xx}+w_{yy}+cw_y-\mu_0 w=  (\frac{F'(\pm1)-F'(u)}{\pm 1-u} -\mu_0)
\cdot w \ge 0, \quad (x, y) \in \Omega^{\pm}.
$$
Now we choose two positive constants $\mu_1$ and $\mu_2$ as
$$
\mu_1= \frac{c+\sqrt{ c^2+ 8\mu_0}}{4},  \quad \mu_2=
\frac{-c+\sqrt{ c^2+ 8\mu_0}}{4}.
$$
Note  that
$$
\mu_1=\mu_2 +c/2, \quad  \mu_1^2 +\mu_2^2 =c^2/4 +\mu_0.
$$

For any rectangular domain $D_R:= \{ (x, y):  |x| \le R,\, 0<y<R\}$,
we consider the function
$$
B(x, y)= 4 e^{ -\mu_2R-cy/2} \cosh(\mu_1 y) \cosh(\mu_2 x), \quad
(x, y) \in D_R.
$$
Straight forward computations reveal that
$$
B_{xx}+B_{yy}+cB_{y} -\mu_0 B=0, \,\, \text{ in } \, D_R, \quad B
\ge 1 \text{ on } \, \partial D_R.
$$

Now for any $(x_0, y_0) \in \Omega^{\pm}$, let $R=R(x_0, y_0)$ be
the distance of $(x_0, y_0)$ to $\Omega^0$ and compare $w(x, y)$
with $B(x-x_0, y-y_0)$ in $D_R(x_0, y_0):=\{(x,y):  (x-x_0, y-y_0)
\in D_R\}$.   Then the maximum principle implies that
$$
w(x,y) \le B(x-x_0, y-y_0), \quad (x, y) \in D_R(x_0, y_0).
$$
In particular, we have $ w(x_0, y_0) \le B(0,0)= 4e^{-\mu_2 R}$. In
view of \eqref{slope}, \eqref{width} and the definition of $k_1,
k_2$, we know that, for  $R(x, y) \ge K$, there exists   some
constant $\mu_3 \in (0, 1)$ such that  $R(x, y) \ge \mu_3 d(x, y) $
when $|y|>Y_0$ and $R(x, y) \ge \mu_3 |x|$ when $ |y|\le Y_0$.

Hence we  derive
$$
|u^2-1|   \le C_0 e^{ -\nu d(x, y)} , \quad  |y|>Y_0; \quad
|u^2-1|\le C_0e^{-\nu |x| }, \quad |y|\le Y_0.
$$
for $\nu=\mu_2\mu_3$ and some constant $C_0>0$.  Note that $\nu<
\sqrt {\mu_0} \le \min\{ \sqrt {F''(1)}, \sqrt {F''(-1)}  \}$.

Then \eqref{exponential} follows from the  standard estimates for
elliptic equations.

\end{proof}

With the exponential decay of $u$, we can define
\begin{equation}
\rho(y)=\rho(y; u):=\int_{\Er} [\frac{1}{2}( |\nabla_x u|^2 -
u_y^2)+F(u) ]dx, \quad y \in \Er.
\end{equation}

The following Hamiltonian identity holds.

\begin{lemma}\label{hamiltonian}
For any $y_0, y \in \Er$, there holds the following Hamiltonian
identity
\begin{equation}\label{hamiltonian-wave}
\rho(y)-\rho(y_0)= c \int_{y_0}^{y} \int_{\Er} |u_y|^2 dx dy.
\end{equation}
\end{lemma}

\vskip0.2in

\subsection{ Only Case (i) is valid}

\vskip0.2in

 Using  the exponential decay \eqref{exponential}, the
Hamiltonian identity \eqref{hamiltonian-wave} and Lemma
\ref{profile}, we can exclude the Cases (ii)-(iv) in subsection 2.1.

The next lemma further asserts that only the first case is possible.

\begin{lemma}\label{onecase}
Assume that  $u$  is solution to \eqref{eq-wave}, \eqref{monotone}
and \eqref{limit-tw}, and the graph of $y=\gamma(x)$ is the 0-level
set of $u$.  Then

\begin{equation}\label{behavior}
 \lim_{ x \to \infty} \gamma'(x) =\infty,  \,\, \lim_{x \to -\infty} \gamma'(x) = -\infty.
\end{equation}
\end{lemma}

\begin{proof}

In Case (ii),  using the exponential decay \eqref{exponential} and
Lemma \ref{profile},  we can compute straight forwardly
$$
\lim_{y \to \infty} \rho(y)=\lim_{ s \to \infty} \rho(0; u_s)=
\int_{-\infty}^{\infty}  [\frac{1}{2}( |g'|^2(x) +F(g(x)) ]dx =\be
$$
and
$$
\lim_{y_0 \to -\infty} \rho(y_0)=\lim_{ s \to -\infty} \rho(0; u_s)=
\int_{-\infty}^{\infty}  [\frac{1}{2}( |g'|^2(x) +F(g(x)) ]dx= \be.
$$
Then the Hamiltonian identity \eqref{hamiltonian-wave} leads to
$$
 \int_{\Er} \int_{\Er} |u_y|^2 dx dy=0.
 $$
 This is a contradiction.   Case (iii) can be excluded similarly.

In Case (iv), we have

$$
\lim_{y \to \infty} \rho(y)=0
$$
and
\begin{equation*}
\begin{split}
\lim_{y_0 \to -\infty} \rho(y_0) &= \lim_{y_0 \to -\infty}
\int_{-\infty}^0 [\frac{1}{2}( |\nabla_x u|^2 - u_y^2)+F(u) ]dx \\
& + \lim_{y_0 \to -\infty} \int_0^{\infty} [\frac{1}{2}( |\nabla_x
u|^2 - u_y^2)+F(u) ]dx \\
&=2\int_{-\infty}^{\infty}  [\frac{1}{2}( |g'|^2(x) +F(g(x)) ]dx=
2\be.
 \end{split}
 \end{equation*}

 This leads to
$$
 \int_{\Er} \int_{\Er} |u_y|^2 dx dy =-\frac{2\be}{c}<0.
 $$
 This is a contradiction, and the lemma is proven.
 \end{proof}

\vskip0.2in
\subsection{The level set  curve  is asymptotically hyperbolic
cosine}

\vskip0.2in

In this subsection, we shall show that the 0-level set $y=\gamma(x)$
of $u$ is  asymptotically hyperbolic cosine. It is more convenient
to write the level set as the graph of functions $ x=k_1(y), \,
x=k_2(y)$ for $y >Y_0$ and show that they are logarithmic. In the
previous subsection, we have already derived properties  for
$y=\gamma(x)$ which can be rewritten for $x=k_{i}(y)$ as follows
\begin{equation}\label{prelim}
\left\{
 \begin{split}
 &k_1'(y)<0, \quad k_2'(y)>0,  \quad \text{ for } \, y  > y_0 \\
 &\lim_{y \to \infty} k_1(y)=-\infty,  \quad \lim_{y \to \infty}
 k_2(y)=\infty \\
 &\lim_{y \to \infty} k_1'(y)=\lim_{y \to \infty} k_2'(y)=0
 \end{split}
 \right.
 \end{equation}
We shall prove the following asymptotical formulas for $k_i(y), i=1,
2.$
\begin{lemma}\label{lemma2.7}
There holds
\begin{equation}\label{asymptotics}
\left\{
\begin{split}
&k_1(y)= -\frac{1}{2\mu} ln(y)+ C_1+o(1), \quad \text{ as }\,  y \to \infty \\
&k_2 (y)=\frac{1}{2\mu} ln(y)+C_2+o(1),\quad \text{ as } \, y \to
\infty
\end{split}
\right.
\end{equation}
for some constants $C_1, C_2$, where $\mu=\sqrt{F''(1)}>0$.
\end{lemma}

\subsubsection{ A standard profile with  two transition layers}

\vskip0.2in

 The proof of  Lemma \ref{lemma2.7}  follows the main ideas of
\cite{chen} in the derivation of similar formula for axially
symmetric traveling wave solutions. Instead of only dealing with one
unknown function in \cite{chen}, here we need to consider the
coupled functions $k_i(y),\, i=1, 2$.  We shall approximate $u(x,
y)$ as functions of $x$ by a family of standard profiles  of two
transition layers for $y$ sufficiently large. Namely, for $l_1<l_2$
and $2l=l_2-l_1$ sufficiently large, we define a continuous and
piecewise smooth function $\phi=\phi(l_1, l_2, x)$ so that it is the
solution of one dimensional Allen-Cahn equation in three segments of
$\Er$:
\begin{equation}
\left \{
\begin{split}
&\phi'' -F'(\phi)=0, \quad  x \in (-\infty, l_1)\cup (l_1, l_2) \cup (l_2, \infty)\\
&\phi(x)>0, \quad  x \in (l_1, l_2); \quad \phi(x)<0, \quad  x \in
(-\infty, l_1)\cup (l_2, \infty)\\
&\phi(l_1)=\phi(l_2)=0,  \quad \lim_{x \to  \pm \infty} \phi(x) =-1
\end{split}
\right.
\end{equation}

Below we collect some basic facts about $\phi=\phi(l_1, l_2, x)$ and
related functions. Indeed, $\phi(l_1, l_2, x)=g(l_2-x)$ for $x>l_2$
and $\phi(l_1, l_2, x)=g(x-l_1)$ for $x < l_1$.  For $ x \in (l_1,
l_2)$, $\phi(l_1, l_2, x)= g(l, x-(l_1+l_2)/2)$  where  $ g(l,
x)=g(l, -x) $  can  be solved explicitly by
\begin{equation}\nonumber
\begin{split}
& g_x^2(l, x)=  2 F(g(l,x))- 2F(g(l,0)), \quad x \in (-l, l)\\
& \int_{g(l,x)}^{g(l, 0)} \frac{ds}{\sqrt{
2(\bigl(F(s)-F(g(l,0))\bigr)}} =x, \quad x \in (0, l)
\end{split}
\end{equation}
where $0<g(l, 0)<1$.

 Note that  elementary  computations can lead to $  \lim_{l \to
\infty} g(l,0) =1$ and
\begin{equation*}
l=\int_{0}^{g(l, 0)} \frac{ds}{\sqrt{ 2(F(s)-F(g(l,0))}}=
-\frac{\ln( 1-g(l,0))}{\mu} +A_1 +o(1)
\end{equation*}
as $l \to \infty$, where $A_1$ is a constant depending only on $F$.
It is also easy to see that $g(l, x)$ is the minimizer of
$$
\bE_l(v):=\int_{-l}^{l} [\frac{1}{2} |v'|^2 + F(v)]dx
$$
in $\calh_l:=\{ v \in H^1_0([-l, l]): 0\le v \le 1,
\,\,v(-l)=v(l)=0\}$ when $l$ is sufficiently large.

If we denote
\begin{equation*}
E(l):= \bE_l(g(l, \cdot))=\bE(\phi(-l, l, \cdot))-\be,
\end{equation*}
then $ E(l)=\be+o(1)$ and
\begin{equation}\label{energy-l}
 \quad E_l:=\frac{\partial E(l)}{\partial l}= |g'(0)|^2-g_x^2(l, l)=2F(g(l,0))= 2\be A e^{
 -2 \mu l+o(1)}
\end{equation}
where $A$ is a positive constant depending only on $F$ (see
\cite{chen}).

 For a  piecewise continuous  function $\psi(x)$  with possible
 jump discontinuities  at $ x=l_1, l_2$, we define
\begin{equation*}
\begin{split}
&\hat{ \psi}= \psi(l_1+)-\phi(l_1-), \quad
\check{\psi}=\psi(l_2+)-\phi(l_2-),\\
& \tilde \psi=\frac{1}{2} \bigl(  \psi(l_1-)+\phi(l_1+)\bigr) ,
\quad \bar{\psi}=\frac{1}{2}\bigl( \psi(l_2+)+\phi(l_2-)\bigr).
\end{split}
\end{equation*}
Note that $E_l= -\hat{\pu^2}=-\hat{\phi_{x}^2}=\check{
\phi_{l_2}^2}=\check{\phi_{x}^2}$.

 We also use the  norm and inner product of
$L^2(\Er)$, i.e.,
$$
\langle\psi_1, \psi_2\rangle  :=\int_{\Er} \psi_1 \psi_2 dx, \quad
\|\psi\|^2:=\langle\psi, \psi\rangle .
$$

Now we state  the following lemma.
\begin{lemma}
For $l=(l_2-l_1)/2>0$,  $\p$ is smooth  except at $ x=l_1, l_2$ and
$$
\phi_{l_1}\le 0, \quad \phi_{l_2} \ge 0, \quad
\|\pr\|^2=E(l)+o(1)=\be+o(1), \quad i=1, 2.
$$
Furthermore, there exists a constant $C>0$ such that   $ \forall
l>1$
\begin{equation*}
\begin{split}
& \sum_{i=1, 2} \|\pr\|_{L^1(\Er)} + \sum_{i, j=1,
2}\|\ps\|_{L^1(\Er)}+\sum_{i, j, k=1, 2} \|\pt\|_{L^1(\Er)} \le C;\\
& \sum_{i=1, 2} \|\pr\| + \sum_{i, j=1, 2}\|\ps\|+\sum_{i, j,
k=1, 2} \|\pt\| \le C\\
& \sum_{i=1, 2} (|\hat{ \pr}|+|\check{ \pr} |)+\sum_{i, j, k=1, 2}
(|\hat{ \ps}|+|\check{ \ps} |)+ \sum_{i=1, 2}  (|\hat{
\phi_{xl_i}}|+|\check{ \phi_{xl_i} }|) \le C \cdot E_{l}\\
& |<\phi_{l_1}, \phi_{l_2}>| + |E_{ll}|+|E_{lll}||  \le C \cdot
E_{l}
\end{split}
\end{equation*}
\end{lemma}

\vskip0.2in

\subsubsection{  Derivation of ordinary differential equations  for
$l_i, i=1, 2$}

\vskip0.2in

Now, for $y>Y_1$ sufficiently large,   we can  choose a unique pair
 $l_1(y) <l_2(y)$ so that
\begin{equation}\label{min-l_i}
||u(\cdot,y)- \phi(l_1(y), l_2(y), \cdot)||_{L^2(\Er)}=
\inf_{l_1<l_2} \{||u(\cdot, y)-\phi(l_1, l_2, \cdot)|| \}.
\end{equation}

As we shall show,  the asymptotical behavior of $u(x,y)$ near
$y=\infty$ can be accurately described by the dynamics of $l_i(y),
i=1, 2$. (See, e.g., \cite{chen} Section 6.1 for an intuitive
explanation for the case $l_1=-l_2$ by using invariant manifold and
center manifold terminology.)

 Let
$$
v(x, y)= u(x, y)-\phi(l_1(y), l_2(y), x), \quad x \in \Er, \quad y
\ge Y_1.
$$
In view of Lemma \ref{profile}, Lemma \ref{exponential1} and Lemma
\ref{onecase}, we see that
\begin{equation}\label{difference}
k_1(y)-l_1(y) \to 0, \quad k_2(y)-l_2(y) \to 0, \quad  \|v(\cdot,
y)\| \to 0, \quad \text{ as } y \to \infty.
\end{equation}

Moreover, using the implicit function theorem, one can see that for
$y>Y_1$ sufficiently large, the functions $l_{i} (y), i=1, 2$ are
smooth and satisfies
\begin{equation}\label{limitofl}
l_1'(y) < 0, \,\,\,  l_2'(y)>0, \,\, \quad  \lim_{y \to \infty}
l_i'(y)=0, \, \, i=1, 2.
\end{equation}
(See, e.g.,  \cite{chen} Lemma 6.2 for a similar statement for the
case  $l_1=-l_2$.)

It is also obvious that
\begin{equation}\label{limitofv}
\lim_{ y \to \infty} \|(|v|+|\nabla v|)\|_{L^{\infty}(\Er)}=0.
\end{equation}

From \eqref{min-l_i}  it is easy to see that
\begin{equation}\label{relation1}
\langle v(\cdot, y), \phi_{l_i} (l_1(y), l_2(y), \cdot ) \rangle=0,
\,\, i=1, 2, \quad y \ge Y_1
\end{equation}

Differentiating the above identities  with respect to $y$ and
dropping the variables  of functions for  the simplicity of
notation, we obtain
\begin{equation}\label{relation2}
\langle v_y, \pr\rangle +  \sum_{j=1, 2} \langle v, \ps\rangle l_j'
-\hat{ v \pr} l_1'+ \check{v \pr}l_2'=0, \quad i =1, 2.
\end{equation}

Differentiating \eqref{relation2} for $i=1$   with respect to $y$,
we have
\begin{equation*}\label{secondderivative}
\begin{split}
& \langle v_{yy}, \pu\rangle  +\langle v_y, \pv\rangle  l_1'+\langle
v_y, \pw\rangle  l_2'-\hat{ v_y \pu}
l_1' +\check{ v_y \pu }l_2'\\
&+\langle v_y, \pv\rangle  l_1'+ [ \langle  v, \pvv\rangle  l_1'+
\langle v, \pvw\rangle  l_2'-\hat{v \pv}
l_1'+ \check{v \pv} l_2'] l_1' + \langle  v, \pv\rangle  l_1''\\
&+ \langle  v_y, \pw\rangle  l_2'+ [ \langle  v, \pvw\rangle  l_1'+
\langle v, \pww\rangle  l_2'- \hat{v \pw}
l_1'+ \check{ v \pw} l_2' ]l_2'+ \langle v, \pw\rangle  l_2''\\
&-[ \hat{ v_y \pu} + \hat{v \pv} \cdot l_1'+ \hat{v \pw} \cdot l_2']
\cdot l_1' - \hat{ v \pu} l_1''\\
&+[ \check{ v_y \pu} + \check{ v \pv} \cdot l_1'+\check{ v \pw}
\cdot l_2'] l_2'+  \check{ v \phi_{l_2l_2}}  l_2''=0.
\end{split}
\end{equation*}
  This leads to
\begin{equation}\label{projection}
 |\langle v_{yy}, \pu \rangle | +| \langle v_y, \pu \rangle | = o(1)
 ( |l_1'|+|l_2'|+|l_1''|+|l_2''|), \quad \text{as} \,\, y \to \infty.
 \end{equation}
 Similar computations can also be done for $i=2$.

Now,  using  equation \eqref{eq-wave} we derive
\begin{equation}\label{eq-v}
v_{xx}+ v_{yy}+c v_{y} -\bigl((F'(v+\phi)-F'(\phi)\bigr)+
\phi_{yy}+c\phi_{y}=0, \quad (x, y)  \in  \Er^2 \setminus \Gamma,
y>Y_1.
\end{equation}
where
\begin{equation*}
\begin{split}
&\phi{_y}=\pu \cdot l_1'+\phi_{l_2} \cdot l_2'\\
&\phi_{yy}=\pv (l_1')^2+ 2\pw (l_1' l_2')+ \phi_{l_2l_2}
(l_2')^2+\pu \cdot l_1''+\phi_{l_2} \cdot l_2''.
\end{split}
\end{equation*}

Multiplying \eqref{eq-v} by $\phi_{l_1}$ and integrating over $\Er$,
we obtain
\begin{equation*}
\begin{split}
&(cl_1'+l_1'')  \|\pu\|^2+  \langle \pu, \phi_{l_2} \rangle
(cl_2'+l_2'')
+ \langle \pv, \pu \rangle (l_1')^2 \\
&+ \langle \pw, \pu \rangle (l_1' l_2')+ \langle \phi_{l_2l_2} , \pu
\rangle (l_2')^2 = J_{1, 1} + J_{1, 2}-J_{1, 3}
\end{split}
\end{equation*}
where
\begin{equation*}
\begin{split}
&J_{1,1}= \langle F''(\phi) v-v_{xx}, \pu \rangle; \\
&J_{1,2}=\langle F'(v+\phi)-F'(\phi)-F''(\phi) v, \pu \rangle;\\
&J_{1, 3} = \langle v_{yy}+ cv_{y},  \pu \rangle.
\end{split}
\end{equation*}

Using $E_l= -\hat{\phi_{l_1}^2}=\check{\phi_{l_1}^2}$, it  can be
computed that
\begin{equation}\label{J1}
J_{1,1}= \hat{ v_x \pu} -\hat{ v \phi_{l_1x} } +\check{ v_x \pu}
-\check{ v \phi_{l_1 x} }= -E_l \bigl(1+ O( |v_x|+|v|)
\bigr)=E_l(1+o(1)).
\end{equation}
Here we have used  \eqref{limitofv}, the fact
$$
\hat{v_x}=-\hat{\phi_{x}}= \hat{\pu}, \quad
\check{v_x}=-\check{\phi_{x}} =-\check{\pu}
$$ and
\begin{equation*}
\begin{split}
& \hat{v_x \phi_{l_1}} = \hat{ v_x} \tilde{\pu} + \tilde{v_x}
\hat{\pu} = \frac{1}{2} \hat{\pu^2}+\tilde{v_x }\hat{\pu}\\
&\check{v_x \phi_{l_1}} = \check{ v_x} \bar{\pu} + \bar{v_x}
\check{\pu} = -\frac{1}{2} \check{\pu^2}+\bar{v_x }\check{\pu}
\end{split}
\end{equation*}

On the other hand,   we have
\begin{equation}\label{J2}
J_{1, 2}= O(1) \langle v^2, \pu \rangle.
\end{equation}
In view of \eqref{projection}, \eqref{J1} and \eqref{J2},  we obtain
\begin{equation}\label{ode1}
cl_1'+l_1''+o(1) (cl_2'+l_2'')= -\frac{E_l}{\be} \bigl(1+o(1)\bigr)
+ o(1) \bigl( l_2'-l_1' \bigr)+ O(1) \|v\| (|l_1''|+|l_2''|) + O(1)
\langle v^2, \pu \rangle.
\end{equation}

Similarly we can obtain
\begin{equation}\label{ode2}
cl_2'+l_2'' +o(1) (cl_1'+l_1'')= \frac{E_l}{\be} \bigl(1+o(1)\bigr)
+ o(1) ( l_2'-l_1')+ O(1) \|v\| (|l_1''|+|l_2''|)+ O(1) \langle v^2,
\phi_{l_2} \rangle.
\end{equation}

Next we shall estimate $ \|v\|$.

\vskip0.2in
\subsubsection{ Estimate of $ \|v\|$ }

\vskip0.2in

We compute
\begin{equation*}
\begin{split}
&\frac{1}{2} \bigl(c\frac{d}{dy}+ \frac{d^2}{dy^2} \bigr) \|v^2\|-
 \langle cv_y+v_{yy}, v \rangle  - \|v_y\|^2\\
 &=-\hat{ v v_y} l_1'+
\check{vv_y} l_2' -\tilde{v} [ \hat{\pu} (l_1')^2+ \hat{\phi_{l_2}}
l_1'l_2'] - \bar{v}[ \check{\phi_{l_2}} l_1' l_2'+
\check{\phi_{l_2}}
(l_2')^2]\\
&= o(1) E_l [(l_1')^2 +(l_2')^2].
\end{split}
\end{equation*}
Here we have used
$$
\hat {v_y} =-\hat{\phi_x} l_1', \quad \check{v_y}= \check{\phi_x}
l_2'.
$$

Due to the non-degeneracy  and stability  property of $g$ in $\Er$,
there holds
\begin{equation}\label{nondege}
\|\psi_x\|^2 +\langle  F''(\phi) \psi, \psi \rangle \ge 2\nu
\|\psi\|^2+ | \tilde{\psi}|^2+ |\bar{\psi}|^2, \quad \forall  \psi
\perp \pr, \, \, i =1, 2
\end{equation}
for some constant $\nu>0$ when $2l=l_2-l_1$ is sufficiently large.
(See also \cite{chen}  Lemma 6.3.)

Multiplying  \eqref{eq-v} by $v$ and integrating on $\Er$, we can
obtain
\begin{equation}\nonumber
\begin{split}
 \langle cv_y+ v_{yy},v \rangle &= \langle F'(v+\phi)-F'(\phi)
-v_{xx} , v \rangle - \langle c \phi_y + \phi_{yy},  v \rangle \\
& \ge 2\nu ( \|v\|^2+ |\tilde{v}|^2 + |\bar{v}|^2 ) - \hat{ v
\phi_x} -
\check{ v \phi_x}\\
&- \langle \pv, v \rangle (l_1')^2 -2 \langle \pw, v \rangle
l_1'l_2' -\langle \phi_{l_2l_2}, v\rangle (l_2')^2\\
&\ge \nu  \|v\|^2 +O(1) E_l^2 +  O(1) [(l_1')^2 +(l_2')^2]^2
\end{split}
\end{equation}
for $y>Y_1$  sufficiently large.

Hence, we derive
\begin{equation}\label{ode3}
\frac{1}{2} \bigl(c\frac{d}{dy}+ \frac{d^2}{dy^2} \bigr) \|v\|^2
-\nu \|v\|^2 \ge  - M_3 \bigl(E_l^2+  (l')^4 \bigr), \quad y > Y_1
\end{equation}
for  some positive constant $M_3$ sufficiently large.

Let $\kappa_1<0< \kappa_2$ be the two roots of the characteristic
equation  $\kappa^2 +c \kappa -2\nu=0$ associated with the operator
on the left hand side of \eqref{ode3}.  Hence, by the maximum
principle for  second order ordinary differential equations, we have
\begin{equation*}
\|v\|^2 \le M_1 e^{\kappa_1 (y-Y_1)} + M_2\ int_{y}^{\infty}
\bigl(E_l^2+  (l')^4 \bigr) e^{\kappa_1(z-y)}dz+  M_2 \int_{Y_1}^{y}
\bigl(E_l^2+  (l')^4 \bigr)  e^{\kappa_2(z-y)}dz
\end{equation*}
for some  positive constants $M_1, M_2$.
 It is easy to see that
$$
\frac{d}{dz} [\bigl(E_l^2+  (l')^4 \bigr) e^{ \frac{1}{2}
\kappa_1(z-y)}] <0, \quad y > Y_1,
$$
and
$$
\frac{d}{dz} [\bigl(E_l^2+  (l')^4 \bigr) e^{ \frac{1}{2}
\kappa_2(z-y)}] >0, \quad y > Y_1
$$
when $Y_1$ is sufficiently large.

Hence we derive
\begin{equation*}
\begin{split}
& \|v\|^2 \le  M_1 e^{\kappa_1 y}+ M_2 \int_{y}^{\infty}
[\bigl(E_l^2+ (l')^4 \bigr) e^{ \frac{1}{2} \kappa_1(z-y)}] \cdot
e^{ \frac{1}{2} \kappa_1(z-y)}dz \\
&+ M_2 \int_{Y_1}^{y}  [\bigl(E_l^2+  (l')^4 \bigr) e^{ \frac{1}{2}
\kappa_2(z-y)}] \cdot e^{ \frac{1}{2} \kappa_2(z-y)} dz\\
& \le  M_1 e^{\kappa_1 y} + M_0  \bigl(E_l^2+  (l')^4 \bigr)
\end{split}
\end{equation*}
for some positive constant $M_1, M_0$.

\vskip0.2in

 \subsubsection{ Derivation of asymptotic formula for
$2l=l_2-l_1$.}

\vskip0.2in

 Now we can write \eqref{ode1}, \eqref{ode2} as

\begin{equation}\label{ode5}
\left\{
\begin{split}
&cl_1'+l_1''= -A  e^{-2\mu l(y)} \bigl(1+o(1)\bigr) + o(1) \bigl(
l_2'-l_1' \bigr)+ o(1) l_2''\\
& cl_2'+l_2''=  A  e^{-2\mu l(y)} \bigl(1+o(1)\bigr) + o(1) (
l_2'-l_1')+ o(1) l_1''.
\end{split}
\right.
\end{equation}
From \eqref{ode5} we can deduce
\begin{equation}\label{single}
\bigl( c+o(1)  \bigr) l'+l''= \bigl(A+o(1)\bigr) e^{-2\mu l(y)}.
\end{equation}
As in \cite{chen}, we can define $Q(y)= e^{2\mu l(y)}$, which
satisfies
\begin{equation}\nonumber
2\mu A +o(1) = \bigl( c+ o(1) \bigr) Q'+ Q''.
\end{equation}
Solving this equation explicitly, we obtain
$$
Q'(y)= \frac{ 2\mu A}{c}+ o(1), \quad Q(y)= \frac{ 2\mu A}{c}  y+
o(y).
$$
Hence we derive
\begin{equation}\nonumber
\begin{split}
&l(y)= \frac{1}{2\mu} \ln(y) +   \frac{1}{2\mu}  \ln( \frac{ 2\mu
A}{c} ) +o(1)\\
& l'(y)= \frac{1+o(1) }{2\mu y},  \quad E_l=\frac{ c+o(1)}{2\mu A y}, \quad l''(y)=\frac{o(1) }{2\mu y} \\
& 0> l_1'(y) \ge  -2 l'(y) =- \frac{1+o(1) }{\mu y}, \quad 0<
l_2'(y) \le 2 l'(y) = \frac{1+o(1) }{\mu y}.
\end{split}
\end{equation}

Hence, we obtain an explicit  estimate for $\|v\|$ in term of $y$
\begin{equation}\label{v}
\|v\|^2  \le  \frac{O(1)}{y^{2}}.
\end{equation}

\vskip0.2in

 \subsubsection{ Derivation of asymptotical formulas for
$l_i, i=1, 2$.}

\vskip0.2in

Now we shall examine more carefully the ordinary differential
equations \eqref{ode1} and \eqref{ode2}.

Define
$$
w(x, y):= u(x, y )-g(x-l_1(y)), \,\,\,  (x, y) \in D:=\{ |x-l_1(y)|
\le l(y)/2, \,\, y \ge Y_1\}
$$
Then $w$ satisfies  a similar equation as \eqref{eq-v}  in $D$ with
$v$ replaced by $w$:
\begin{equation}\label{eq-w}
\begin{split}
&w_{xx}+ w_{yy}+c w_{y} -\bigl((F'(w+g(x-l_1))-F'(g(x-l_1))\bigr)\\
&+ g''(x-l_1) (l_1')^2-g'(x-l_1) (c l_1' + l_1'')=0, \quad (x, y)
\in D.
\end{split}
\end{equation}

It is also easy to see that
\begin{equation}\nonumber
\|w(\cdot, \cdot-l_1)\|_{L^2([l_1-2 , l_1+ 2])} \le \|v\|+
\|\phi-g(\cdot-l_1)\|_{L^2([l_1-2, l_1+2])} \le \|v\|+ O(1) E_l \le
O(1)y^{-1}.
\end{equation}
and
\begin{equation}\nonumber
\|w(\cdot, \cdot-l_1)\|_{L^p([l_1-2 , l_1+ 2])}  \le O(1) \cdot
y^{-1}
\end{equation}
for any $p>2$.

Then, for any fix $y_0>Y_1$, using the standard  $L^p$ interior
estimate of elliptic equations for \eqref{eq-w}  and the Sobolev
inequality in $ [l_1(y_0)-2 , l_1(y_0)+ 2]\times [y_0-2, y_0+2]$, we
obtain
$$
\|\nabla w(\cdot, y)\|_{L^\infty ([l_1-1 , l_1+ 1])} \le O(1) \cdot
y^{-1}, \quad  y \ge Y_1
$$
and hence
$$
\|\nabla v(\cdot, y)\|_{L^\infty ([l_1-1 , l_1+ 1])}+ \| v(\cdot,
y)\|_{L^\infty ([l_1-1 , l_1+ 1])} \le O(1) \cdot y^{-1}, \quad  y
\ge Y_1.
$$
Using the standard $L^p$ interior estimate of elliptic equations for
\eqref{eq-v} outside $\bar \Gamma:=\{ (x, y):
|x-l_1(y)|+|x-l_2(y|<1, y >Y_1\}$  as well as the above estimate in
$\bar \Gamma$, we can also obtain
$$
\|\nabla v(\cdot, y) \| +  \|\nabla v(\cdot, y)\|_{L^\infty(\Er)}
+\|v(\cdot, y)\|_{L^\infty(\Er)} \le O(1) \cdot y^{-1}, \,\, y \ge
Y_1.
$$

Now we use re-examine \eqref{secondderivative}, \eqref{projection},
\eqref{J1} and \eqref{J2},  and obtain
\begin{equation}\label{ode1-final}
cl_1'+l_1''= -\frac{E_l}{\be}+O(1) \cdot  y^{-2}, \quad y>Y_1.
\end{equation}
Similarly, we can derive
\begin{equation}\label{ode2-final}
cl_2'+l_2''= \frac{E_l}{\be}+O(1) \cdot  y^{-2}, \quad y>Y_1.
\end{equation}
Hence, we have
\begin{equation*}\label{ode3-final}
c(l_1+l_2)'+(l_1+l_2)''=O(1)  \cdot y^{-2}, \quad y>Y_1
\end{equation*}
and therefore
 $$
l_1+l_2= O(1) \cdot y^{-1}, \quad y>Y_1.
$$

 This leads to
\begin{equation}\label{asymptotics-l}
\left\{
\begin{split}
&l_1(y)= -\frac{1}{2\mu} \ln(y) - \frac{1}{2\mu}  \ln( \frac{ 2\mu
A}{c} ) +B+o(1)\\
&l_2(y)= \frac{1}{2\mu} \ln(y) +   \frac{1}{2\mu}  \ln( \frac{ 2\mu
A}{c} ) +B +o(1)
\end{split}
\right.
\end{equation}
for some constant $B$.  Lemma 2.7  then follows directly with
$$C_1=- \frac{1}{2\mu}  \ln( \frac{ 2\mu
A}{c} ) +B,\quad C_2=-\frac{1}{2\mu}  \ln( \frac{ 2\mu A}{c} ) +B.
$$

\subsection{The moving plane procedure}

\vskip0.2in

 In this subsection, we shall use the moving plane method
to finish the proof of Theorem 1.1. Due to the fact that the
asymptotical behavior of $u$ is not homogeneous near infinity, in
particular, there is a transition layer along  $\Gamma$, the classic
moving plane method has to be carefully modified.   Indeed, we have
to use the exact asymptotical formulas of the 0-level sets  $x
=k_i(y), i=1, 2$ near infinity as well the asymptotical behavior of
$u$ along these curves.

Define $ u_\lambda(x, y):= u(2\lambda-x, y)$   and
 $w_\lambda:=u_{\lambda} -u$ in $ D_{\lambda}:=\{ (x, y) :  x\ge \lambda, \, y \in \Er \} $.

\begin{lemma}
When $ \lambda $ is sufficiently large,   there holds $w_\lambda
>0$ in $ D_{\lambda}$.
\end{lemma}

\begin{proof}
When $\lambda> \lambda_0$ is sufficiently large, by Lemma 2.7 we
know that
$$
k^{\lambda}_1(y): = 2 \lambda -k_1(y) \ge k_2(y), \quad \forall y
\ge Y_1.
$$

By Lemma 2.3 and Lemma 2.4,  we see that there exist constants $K>0,
Y_2>Y_1$ and $\lambda_1$  sufficiently large such that  when
$\lambda> \lambda_1$,  there hold  $w_{\lambda}
>0$  in $D_{K, Y_2, \lambda} =\{ (x, y) \in D_{\lambda}:  x <
k^{\lambda}_1(y)+K, \, y\ge Y_2\}$ and
 $ u<\alpha^-$  in  $D^c_{K, Y_2, \lambda}: =\{ (x, y) \in
D_{\lambda}:  x > k^{\lambda}_1(y)+K, \, y\ge Y_2 \, \text{  or }
\,\, \forall x \ge \lambda,  \, y \le Y_2 \}$.  Note that $F''(s)
>\mu_0>0$ for $ s \in (-1, \alpha^-]$ by the definition of
$\alpha^-$.

We claim that $w_\lambda \ge 0$ in $D_{\lambda}$ for $\lambda>
\lambda_1$.  If it is not true,  there exists a sequence of points
$\{(x_m, y_m)\}_{m=1}^\infty \in D^c_{K,Y_2, \lambda} $ such that
$$
\lim_{m \to \infty} w_\lambda (x_m, y_m)=  \lim_{m \to \infty}
\bigl( u_\lambda(x_m, y_m)-u(x_m, y_m) \bigr) =\inf_{D^c_{K, Y_2,
\lambda}} w_\lambda(x, y)<0.
$$
It can be seen  that $u_{\lambda}(x_m, y_m)< \alpha^-$ when $m $ is
large enough.  Then we can follow the standard translating arguments
to obtain a contradiction. Define $w^m_\lambda(x, y):= w_\lambda(
x+x_m, y+y_m)$ in $D^c_{K, Y_2, \lambda}-(x_m, y_m)$. Then $
w^{m}_\lambda $ converges to $w^{\infty}_\lambda (x, y)$ in
$C^3_{loc} (D^{\infty})$ for some piecewise Lipschitz  domain $
D^{\infty}$ in $\Er^2$ which contains a small ball centered at the
origin. Furthermore, $w^{\infty}_\lambda$ attains its negative
minimum at the origin and satisfies a linearized equation
\begin{equation}\label{elliptic-w}
w_{xx}+ w_{yy}+cw_{y}-F''(\xi(x, y)) w=0, \quad (x, y) \in D^\infty
\end{equation}
where $\xi(x, y)=su(x, y)+(1-s)u_{\lambda}(x, y)$ for some $s \in
(0, 1)$  and $F''(\xi(0,0))
>\mu_0>0$. This is a contradiction, which leads to the claim.  Then
the lemma follows from the  strong maximum principle (or the Harnack
inequality) applied to an elliptic equation similar to
\eqref{elliptic-w} which is satisfied by $w_\lambda$.

\end{proof}

Now we define
 $$ \Lambda=\inf\{ \lambda :   u_\lambda(x, y) > u(x, y), (x, y)
\in D_\lambda\}.
$$

\begin{lemma} There holds
$$  \Lambda =(C_1+C_2)/2 $$
where $C_1, C_2$ are as in Lemma 2.7.
\end{lemma}

\begin{proof}
We shall prove this lemma by contradiction. Suppose the lemma does
not hold.  By Lemma 2.3 and Lemma 2.7, we can easily see that
$\Lambda
>(C_1+C_2)/2$ and $w_{\Lambda}>0, \forall (x, y) \in D_{\Lambda}$.
 Then there exists a sequence of numbers $\{\lambda_m\} $
 such that $\lambda_m <
\Lambda $,  and $\lim_{m \to \infty} \lambda_m =\Lambda$ and the
infimum of $ w_{\lambda_m}$ in $D_{\lambda_m}$ is negative.  Using
Lemma 2.3, Lemma 2.4, Lemma 2.7 and  the translating arguments in
the proof of Lemma 2.9, we can show that  the infimum of $
w_{\lambda_m}$ in $D_{\lambda_m}$ is achieved at a point $(x_m,
y_m)$, i.e.,
\begin{equation}\label{infimum}
w_{\lambda_m} (x_m, y_m)=\inf_{D_{\lambda_m}} w_{\lambda_m}<0.
\end{equation}
Since $w_{\lambda_m}$ satisfies an elliptic equation similar to
\eqref{elliptic-w} with $\xi(x_m, y_m)= s u(x_m, y_m)+(1-s)
u_{\lambda_m}(x_m, y_m)$ for some $s \in (0,1)$,  by the strong
maximum principle we know that $u (x_m, y_m)> \alpha^-$ and hence
 $y_m> -K_1$  and $x_m-k_1(y_m)<K$  if $y_m>Y_1$ for some constant $K, K_1>0$
independent of $m$.  By Lemma 2.3, Lemma 2.7 and the assumption
$\Lambda> (C_1 +C_2)/2$,  we know $y_m <K_2$ for some constant $K_2$
independent of $m$. Therefore there exists a subsequence of $\{m \}$
(still denoted by the same) such that $(x_m, y_m)$ converges to
$(x_0, y_0) \in D_{\Lambda}$  and $w_{\lambda_m} $ converges  to
$w_{\Lambda}$ in $C^3_{loc}(D_{\Lambda})$ as well as  in
$C^3(B_1(x_0, y_0)\cap \bar{D_{\Lambda}})$. It is easy to see that
$\frac{\partial}{\partial x} w_{\Lambda}(x_0, y_0)=0$. Furthermore,
$w_{\Lambda} $ satisfies an elliptic equation similar to
\eqref{elliptic-w} in $D_{\Lambda}$, hence $(x_0, y_0)$ must be on
the  boundary  of $D_{\Lambda}$.   Then by the Hopf Lemma, we have
$\frac{\partial}{\partial x} w_{\Lambda}(x_0, y_0)<0$. This is a
contradiction, which proves the lemma.
\end{proof}

We note that  $ u_\Lambda \ge u$ in $D_{\Lambda}$ and
$u_{x}(\lambda, y)=-\frac{1}{2} \frac{\partial }{\partial x}
w_{\lambda}(\lambda, y)>0, \forall y \in \Er$ when $\lambda>
\Lambda$.  Similarly, we can use the moving plane method from the
left, i.e., repeating the above procedure for $w_{\lambda}:=$ in
$D^-_{\lambda}:=\{ (x, y): x<\lambda\}$, and conclude $ u_\Lambda
\ge u$ in $D^-_{\Lambda}$. Therefore, Theorem 1.1 is proven.

The uniqueness of the traveling wave solutions  (up to translation)
 still  remains an open question.

\section{Classification of Traveling Wave Solutions for the
Unbalanced Allen-Cahn Equation in $\Er^2$}

In this section, we shall assume that the double well potential $F$
in the Allen-Cahn equation \eqref{allen-cahn} is unbalanced, i.e.,
$F$ satisfies \eqref{doublewell} and  $F(1)>F(-1)=0$. In this case,
 one dimensional traveling wave solution $g$ to \eqref{1d-ac} exists
 for a unique  $c_0>0$  which only depends  on
 $F$, and  $g$  is unique up to translation.
 It is easy to see that a rotation of the trivial extension of  $g$
 to two dimensional plane is also a traveling wave solution  of \eqref{eq-wave}
 for some constant $c$. Indeed,  if  $\alpha \not =\pi/2, 3\pi/2$, then $u(x, y)= g(y
 \cos{\alpha}-\sin{\alpha})$  satisfies \eqref{eq-wave}
 with $c=\frac{c_0}{\cos{\alpha}}$.
 In addition to the one dimensional traveling wave solutions,
 so called $V$-shaped two dimensional traveling wave solutions are  shown
 to exist in \cite{hmr2}, \cite{nt1}.  These solutions are monotone in $y$  and even with respect to $x$
 after a proper translation. The $0$-level set  of such
 solutions are asymptotically two straight rays  forming a shape of
 $V$. The existence result may be stated as follows.

 \begin{theoremb} [Hamel, Monneau, Roquejoffre  \cite{hmr2};  Ninomiya,  Taniguchi, \cite{nt1};
 2005]
For each $\alpha \in [0, \pi/2)$
 there exists a  solution $u_{\alpha}$  of \eqref{eq-wave}, \eqref{monotone} and \eqref{limit-tw} such that
 $c=\frac{c_0}{\cos{\alpha}}$ and $ u_{\alpha} $  is even in $ x$
and decreasing in $|x|$.  The $0$-level set of $u$  is a globally
Lipschitz graph of  $y = k(x) $  and $k(x)= (\tan{\alpha}+o(1)) |x|$
as $|x| $ goes to infinity.
\end{theoremb}

Furthermore, it is shown in \cite{hmr3} that such $V$ shaped
traveling wave solutions are unique for each $\alpha \in [0,
\pi/2)$. Indeed, the following classification theorem  is proven.

 \begin{theoremc} [Hamel, Monneau, Roquejoffre, 2006]
 Suppose  $u$ is a  solution to \eqref{eq-wave}, \eqref{monotone} and
\eqref{limit-tw}.  Assume further that  there exists a
   globally Lipschitz function $\psi$ such that
\begin{equation}\label{limit-hamel}
\left\{
\begin{split}
  & \liminf_{A \to +\infty, y \ge A +\psi(x)} u(x,y)>0,\\
  & \limsup_{A \to -\infty, y \le A +\psi(x)}u(x, y) <0.
  \end{split}
  \right.
  \end{equation}

Then  $c\ge c_0$,  and   $u$ must be either planar, i.e. $u(x, y)=
g(  y \cos{\alpha}\pm x \sin{\alpha } +b)$ with
$\alpha=\cos^{-1}(c_0/c) \in [0, \pi/2)$ and a constant $b$, or  $u$
is the unique even $V$-shaped traveling wave solution  $u_{\alpha}$
(up to translation in $x$ and $y$).
\end{theoremc}

We note that  for $n\ge 3$, similar conic shaped solutions are also
shown to exist in \cite{hmr2},  and the uniqueness of traveling wave
solutions with   $0$-level set being prescribed asymptotical
circular cone is proven in \cite{hmr2}. More complicated pyramidal
traveling wave solutions also exist for $n \ge 3$,  and these
solutions  are  unique when the $0$-level sets are prescribed as
given pyramidal cones at infinity (see \cite{tani1}, \cite{tani2}).

The above classification theorem is very interesting. However,  the
condition \eqref{limit-hamel} is too restrictive. We shall show that
this condition can indeed be dropped.  For this purpose, it suffices
to show that $0$-level set of $u$ must be global Lipschitz, since
the $0$-level set function  $y=k(x)$  can serve as the function
$\psi$ in \eqref{limit-hamel}.

\begin{lemma}
Assume that $u$ is a solution to \eqref{eq-wave}, \eqref{monotone}
and \eqref{limit-tw}, and the graph of $y=k(x)$ is the $0$-level set
of $u$.  Then $k(x) \in C^3(\Er)$ and $|k'(x)| \le C, x \in \Er$ for
some constant $C>0$.
\end{lemma}

\begin{proof}
It is easy to see that $k(x)$ is in $C^3(\Er)$.  We shall prove the
global Lipschitz property by contradiction. Assume that there exists
a sequence $\{x_m\}$ such that $k'(x_m) \to \infty$ as $m $ tends to
infinity. Since $u_x(x, k(x))+ u_y(x, k(x)) k'(x)=0, \forall x \in
\Er $ and  $\nabla u$ is bounded in $\Er^2$,  we derive $ u_y( x_m,
k(x_m)) \to 0$ as $m$ goes to infinity.  We shall investigate the
translation of $u$ along $\bigl(x_m, k(x_m)\bigr)$. Define  $ u^m(x,
y):= u(x+x_m, y+y_m)$. Since $u$ is bounded in $C^{3, \beta}
(\Er^2)$ for some $\beta \in (0, 1)$, it is easy to see that $u^m$
(up to a subsequence ) converges  to $u^*$ in $C^3_{loc}(\Er^2)$,
and $u^*$ satisfies \eqref{eq-wave}. Hence $u^*_y(x, y)$ satisfies
the linearized equation
\begin{equation}
w_{xx}+w_{yy}+cw_y-F''(u^*) w=0, \quad (x, y) \in \Er^2.
\end{equation}
By \eqref{monotone}, we know that $u^*_y(x, y) \ge  0, \forall (x,
y) \in \Er^2$.  Since $ u^*_y(0,0)=\lim_{m \to \infty} u_y(x_m,
k(x_m))=0$,  by the strong maximum principle for elliptic equations
we obtain $u^*_y \equiv 0 $ in $\Er^2$.  Therefore,  $u^*(x,
y)=u^*(x) $ satisfies the one dimensional stationary Allen-Cahn
equation \eqref{allencahn-1d} with $|u^*(x)| \le 1, x \in \Er$ and
$u^*(0)=0$. Then  we have either

\noindent{ \bf Case I:}  $u^*(x) = g_{\alpha}(x \pm K_{\alpha}), x
\in \Er, $ where $g_{\alpha}$ is a periodic solution of the  one
dimensional Allen-Cahn equation
\begin{equation}\label{periodic-ac}
\left \{
\begin{split}
& g_{\alpha}''(x)-F'(g_{\alpha}(x))=0, \quad x \in \Er,\\
& g_{\alpha}'(0)=0, \quad  g_{\alpha}(0)=\alpha,
\end{split}
\right.
\end{equation}
with $\alpha = \max_{\Er} u^*(x) \ge 0$  and $K_{\alpha}$ is the
smallest positive zero of $g_{\alpha}$ if $\alpha>0$; or

\noindent{\bf  Case II:}   $u^*(x) =g_{*}(x\pm K_*)$ where $g_{*}$
satisfies
\begin{equation}\label{limit-ac}
\left \{
\begin{split}
& g_{*}''(x)-F'(g_{*}(x))=0,  \quad |g_{*}(x)| <1, \quad x \in \Er,\\
& g_{*}'(0)=0, \quad  \lim_{|x| \to \infty} g_{*}(x)=1
\end{split}
\right.
\end{equation}
with $K_*>0$ being the only positive zero of $g_{*}$.

We note that $g_0 \equiv 0$ and there is no standing wave solution
to \eqref{1d-ac} with $c_0=0$ in this case.  It is well-known that
$g_{\alpha}$ is unstable in the sense that the linearized operator
$$
\call_{\alpha} \psi:= -\psi''+F''(g_{\alpha}) \psi
$$
has a negative first eigenvalue $-\mu_{\alpha}$ in the periodic
subclass of $H^2(\Er)$ with period $L=L(\alpha)$. It is also
well-known that $g_{*}$ is unstable in the sense that the linearized
operator
$$
\call_{*} \psi:= -\psi''+F''(g_{*}) \psi
$$
has a negative first eigenvalue $-\mu_{*}$ in $H^2(\Er)$. (See,
e.g., \cite{hr1}.)

Now we repeat the computations as in  \eqref{u_y1} in the proof of
Lemma 2.1.,  and obtain
\begin{equation}\label{unbalanced-uy1}
h'(x) =F(1)-c\int_{\Er}u_y^2dy, \quad \forall x \in \Er.
\end{equation}
Hence, for any $a<b$ we have
\begin{equation}\label{unbalanced-uy2}
\int_{a}^{b} \int_{\Er} u_y^2dydx =\frac{1}{c} \bigl(h((a)-h(b)
\bigr)+ \frac{F(1) }{c} ( b-a).
\end{equation}
On the other hand, as in the proof of Lemma 2.2  we define
$$
\bar{h}_m(x)=\int_{-\infty}^0 \frac{\partial u^m}{\partial
x}\frac{\partial u^m}{\partial y } dy.
$$
It is easy to see that $|\bar{h}_m(x)|<C$ for some constant $C$
independent of $ x $ and $ m $. We can  also derive

\begin{equation*}
\bar{h}_m' (x)=-c \int_{-\infty}^0 (\frac{\partial u^m}{\partial
y})^2 dy + \frac{1}{2} (\frac{\partial u^m}{\partial x})^2 (x, 0)
-\frac{1}{2} (\frac{\partial u^m}{\partial y})^2 (x, 0)+F( u^m(x,
0)), \quad  \forall x \in \Er
\end{equation*}

In Case I,  in view of \eqref{unbalanced-uy2} we have for any fix
$R>0$,
\begin{equation*}
\int_{-R}^R \bigl[\frac{1}{2} (\frac{\partial u^m}{\partial x})^2
(x, 0) -\frac{1}{2} (\frac{\partial u^m}{\partial y})^2 (x, 0)+F(
u^m(x, 0))\bigr] dx<C+ 2 F(1) R
\end{equation*}
for some constant $C$ independent of $m, R$.

 Letting $m $ go to
infinity,  from $u^*(x)=g_{\alpha}(x-K_{\alpha})$  we obtain
\begin{equation*}
\int_{-R}^{R} \bigl[\frac{1}{2} |g_{\alpha}'|^2 +F(g_{\alpha})
\bigr] dx \le C+ 2F(1) R.
\end{equation*}
However, by the property  \eqref{doublewell} of $F$, we have
$F(g_{\alpha}) \ge F(\alpha)>F(1)$ in $\Er$ and hence
\begin{equation*}
\int_{-R}^{R} \bigl[\frac{1}{2} |g_{\alpha}'|^2 +F(g_{\alpha})
\bigr] dx > 2F(\alpha) R.
\end{equation*}
This is a contradiction when $R$ is sufficiently large.

In Case II, in view of \ref{unbalanced-uy2}  we have,  for any fix
$R>0$,
\begin{equation}\label{case2}
\begin{split}
&\int_{-R}^R \bigl[\frac{1}{2} (\frac{\partial u^m}{\partial x})^2
(x, 0) -\frac{1}{2} (\frac{\partial u^m}{\partial y})^2 (x, 0)+F(
u^m(x, 0))-F(1) \bigr] dx\\
& \le   \bigl( \bar{h}_m( R)- h(R+x_m)\bigr) - \bigl( \bar{h}_m(
-R)- h(-R+x_m)\bigr).
\end{split}
\end{equation}

We note that for any $x \in \Er$,
 $$
  | h(x+x_m)-\bar{h}_m(x)| \le C \int_{0}^{\infty} u_{y}
\bigl(x+x_m, k(x_m)+y\bigr)dy
 \le C[1-u\bigl(x+x_m, k(x_m)\bigr)]
 $$
 and hence
 $$
 \lim_{x \to \pm \infty} | h(x+x_m)-\bar{h}_m(x)| \le C[1-u^*(x)]
 $$
 for some constant $C$.

 Since
$u^*=g_{*}$,  we have
\begin{equation*}
\lim_{x \to \pm \infty}  [\lim_{m \to \infty} \bigl(
h(x+x_m)-\bar{h}_m(x) \bigr)]=0.
\end{equation*}.

 From  \ref{case2}, by first letting $m$ go to infinity and then letting
$R$ go to infinity, we obtain
\begin{equation*}
\int_{\Er} \bigl[\frac{1}{2} |g_{*}'|^2 +F(g_{*})-F(1) \bigr] dx \le
0.
\end{equation*}

On the other hand, it is easy to see
$$
\frac{1}{2} |g_{*}'|^2-F(g_{*}(x))=-F(1),\quad  \forall x \in \Er.
$$
and hence  $F(g_{*}(x))-F(1)>0,  \forall x \in \Er$. This leads to a
contradiction.
 Therefore, Lemma 3.1 is proven.

\end{proof}

Combining Lemma 3.1 with Theorem C, we immediately have the
following classification theorem for traveling wave solutions of the
unbalanced Allen-Cahn equation.

 \begin{theorem}
 Assume that $F$ is a unbalanced double well potential satisfying \eqref{doublewell} and
 $F(1)>F(-1)=0$.  Suppose  $u$ is a  solution to \eqref{eq-wave}, \eqref{monotone} and
\eqref{limit-tw}. Then  $c\ge c_0$ where $c_0$ is the unique speed
of one dimensional traveling wave solution $g$ as in \eqref{1d-ac},
and $u$ must be either planar, i.e. $u(x, y)= g( y \cos{\alpha} \pm
x \sin{\alpha } +b)$ with $\alpha=\cos^{-1}(c_0/c) \in [0, \pi/2)$
and  $b$ being a constant, or $u$ is the unique even $V$-shaped
traveling wave solution $u_{\alpha}$ (up to a translation in $x$ and
$y$).
\end{theorem}

\section{Traveling Waves Solutions Connecting Other One Dimensional Stationary Solutions}

 If we drop the limit assumption \eqref{limit-tw} and  instead  define
$$
u^\pm (x)= \lim_{y \to \pm \infty} u(x, y), \quad x \in \Er,
$$
then there are  eight possibilities for  the balanced Allen-Cahn
equation:

\begin{equation*}
\begin{aligned}
&\hskip-1.2in  \text{\bf (1) } \quad  u^+ =g(x-K_1), \quad
\quad u^-\equiv -1;\\
&\hskip-1.2in \text{\bf (2) } \quad  u^+ =g(K_1-x),  \quad u^-
\equiv -1;\\
&\hskip-1.2in  \text{\bf (3) } \quad  u^+ \equiv 1, \quad  u^-
=g(x-K_2);\\
&\hskip-1.2in  \text{\bf (4) }  \quad   u^+ \equiv 1, \quad
u^-=g(K_2-x);\\
&\hskip-1.2in  \text{\bf (5) } \quad  u^+=g(x-K_1), \quad
u^-=g(x-K_2), \quad K_1< K_2 ;\\
&\hskip-1.2in  \text{\bf (6) }  \quad   u^+=g(K_1-x), \quad
u^-=g(K_2-x), \quad  K_1> K_2 ;\\
&\hskip-1.2in  \text{\bf (7) }  \quad  u^+ = g_{\alpha}(x-K), \quad
u^-
\equiv -1;\\
&\hskip-1.2in  \text{\bf (8) }  \quad u^+ \equiv 1, \quad u^-=
g_{\alpha}(x-K),
\end{aligned}
\end{equation*}
where $g_{\alpha}(x), \alpha \in [0,1)$ is the periodic solution of
one dimensional Allen-Cahn eqaution  \eqref{periodic-ac} (note that
$g_0\equiv 0$), $K$ is some constant.

Modifying the  arguments  in the proofs of Lemma 2.1-2.5, we can
 exclude  Cases (1)-(6) by showing similar properties for $u$ as in Lemma 2.1-2.5.
   For example,  to
exclude Case (1), we modify \eqref{u_y1} and \eqref{u_y2} as
follows.
\begin{equation}\label{u_y3}
\begin{split}
h'(x)=&\int_{\Er} ( u_{xx}u_y + u_x u_{xy}) dy \\
=& -c\int_{\Er} u_y^2dy + F(u^+)+\frac{1}{2} |u^+_x|^2, \quad
\forall x \in \Er.
\end{split}
\end{equation}

Then
\begin{equation}\label{u_y4}
\int_{a}^{b} \int_{\Er} u_y^2dydx \le \frac{1}{c} \bigl(h((a)-h(b)
+\be \bigr).
\end{equation} and Lemma 2.1 still holds.
The rest will be essentially the same except that the 0-level set is
the graph of a function $y=\gamma(x)$  which is only defined in
$(-\infty, K_1)$.  We  just replace $x \to \infty$ by $ x \to K_1$
in the appropriate places. In this case, we can show Lemma 2.6 and
Lemma 2.7 as well. This leads to a contradiction with $u^+$.  Case
(2) can be similarly excluded.
 Cases (3)-(6) are similar up to Lemma 2.5., and can be excluded
 directly by using the Hamiltonian identity as in the proof of Lemma
 2.6. The details are omitted and  left to the reader.

Therefore, we have the following  nonexistence theorem.

\begin{theorem}
Assume that $F$ is a balanced double well potential, i.e.,  $F$
satisfies \eqref{doublewell} and $F(1)=F(-1)=0$.  Then there exists
no solution to \eqref{eq-wave} and \eqref{monotone} with the limits
 being  one of  the  above Cases (1)-(6), i.e., at least one of $u^+,
 u^-$ being a refection and/or  translation of $g$.
\end{theorem}

For the unbalanced Allen-Cahn equation,  since $g$ is not a
stationary solution of \eqref{allencahn-1d}, there are only four
possibilities for $u^+, u^-$: in addition to Cases (7), (8) listed
above, there are the following two more cases:

\begin{equation*}
\begin{aligned}
&\hskip-2.2in  \text{\bf (9) }  \quad  u^+ = g_{*}(x-K), \quad u^-
\equiv -1;\\
&\hskip-2.2in  \text{\bf (10) }  \quad u^+ \equiv 1, \quad u^-=
g_{*}(x-K),
\end{aligned}
\end{equation*}
where $g_{*}$ is the unique solution to \eqref{limit-ac}.

 In Cases
(7) and (8) there is no difference between the balanced and
unbalanced Allen-Cahn equation  since Case (7) is only involved with
$F(u)$ when $u \le \alpha<1$ and Case (8) is only involved with
$F(u)$ when $u \ge -\alpha>-1$. These cases are indeed of monostable
type.  Case (7) could happen for sufficiently large $c>0$, as shown
in \cite{hr1} for $\alpha \in (0, 1) $ and \cite{hn} for $\alpha=0$
(see also \cite{hm1} for a Bunsen flame model and \cite{mn} for
related results). To be more precise, it is proven in \cite{hr1}
that for $L>L_{min}:=2 \pi \sqrt{-F''(0)}$, there exists a positive
minimum speed $c_{L}>c_0$ such that \eqref{eq-wave} has a
$L$-periodic solution $u_{c, L} (x, y)=u_{c, L} (x, y+L)$ satisfying
the limit condition $u^+=g_{\alpha(L)}, u^- \equiv -1$  if and only
if $c \ge c_L$. Here $\alpha(L)$ can be uniquely determined so that
the period of $ g_{\alpha(L)} $ is $L$. It is also shown in
\cite{hr1} that  Case (9) can happen for sufficiently large $c>0$.
Indeed, there exists $c_{\infty}>c_0$ such that there exists a
solution to \eqref{eq-wave} and \eqref{monotone} with {\it uniform}
limits $u^+, u^-$ being in Case (9) if and only if $c\ge
c_{\infty}$. (See Theorems 1.1 and 1.4 in \cite{hr1}.)

However, Cases (8) and (10) can be excluded by using a generalized
Hamiltonian identity.

\begin{theorem}
 Assume that $F$ is a  double well potential, i.e., $F$
satisfies \eqref{doublewell}.  Then there exists no solution to
\eqref{eq-wave} and \eqref{monotone} with $ u^+ \equiv 1, u^-=
g_{\alpha}(x-K) $  for any constants $\alpha \in [0, \infty), K \in
\Er$. When $F$ is unbalanced, there exists no solution to
\eqref{eq-wave} and \eqref{monotone} with $ u^+ \equiv 1, u^-=
g_{*}(x-K)$ for any constant $K$.
\end{theorem}

\begin{proof}
We just note that as in \eqref{u_y3}, there holds
\begin{equation}\label{u_y5}
\begin{split}
h'(x)=&\int_{\Er} ( u_{xx}u_y + u_x u_{xy}) dy \\
=& -c\int_{\Er} u_y^2dy + F(1)- F(u^-)-\frac{1}{2} |u^-_x|^2, \quad
\forall x \in \Er.
\end{split}
\end{equation}

In Case (8),  $F(u^-) \ge F(\alpha)> F(1)$, then
\begin{equation}\label{u_y6}
\int_{a}^{b} \int_{\Er} u_y^2dydx \le \frac{1}{c} [ h(a)-h(b) +(b-a)
\bigl(F(1)-F(\alpha) \bigr) ].
\end{equation}
This leads to  a contradiction when $b-a$ is  chosen sufficiently
large.

In Case (10), we have
$$
\lim_{ x \to \infty} h(x)=\lim_{ x\to -\infty} h(x) =0.
$$
Then
\begin{equation}\label{u_y7}
c\int_{\Er} \int_{\Er} u_y^2dydx  \le  \int_{\Er} \bigl( F(1)-
F(g_{*})-\frac{1}{2} |g_{*}'|^2 \bigr)dx<0.
\end{equation}
This is a contradiction.  The theorem is proven.

\end{proof}

We remark that  Cases (8) and (10) could happen when the speed $c$
is sufficiently negative.  These cases are similar to (7) or (9)
except that the traveling directions should be reversed.  Another
way to understand these cases is to reverse the spatial direction
$y$ while the speed $c$  is kept positive. However,  the monotone
condition \eqref{monotone} is changed to decreasing in this
approach.

Next, we shall show that when $c>0$ is sufficiently small,
  there is no monotone traveling wave solutions with  limits as in Case (7)
   without requiring solutions being  periodic in $x$ nor the limits being
   uniform in $ x\in \Er$.

\begin{theorem}
Assume that $F$ is a double well potential, i.e., $F$ satisfies
\eqref{doublewell}. Then,  for $L \in [0, \infty)$,  there exists a
constant $c^*_L>0$ such that \eqref{eq-wave} and \eqref{monotone}
has no solution with limits $u^+=g_{\alpha(L)}, u^- \equiv -1$ when
$c< c^*_L$, where $ g_{\alpha(L)} $ is the  solution of
\eqref{periodic-ac} with a period $L$ (we use the convention  that
$\alpha(0)=0$). Similarly there exists a constant $ c^{*}>0$ such
that \eqref{eq-wave} and \eqref{monotone} has no solution with
limits $u^+=g_{*}, u^- \equiv -1$ when $c< c^{*}$ where $g_{*}$  is
the solution of \eqref{limit-ac}.
\end{theorem}

\begin{proof}

We shall first state a gradient estimate  for general traveling wave
solutions to \eqref{eq-wave} as in \cite{mod}, where the same
estimate is proved for  stationary solutions to
\eqref{allencahn-general}.

\begin{proposition}
Assume that $F(s)\ge 0, \forall s \in [-1,1]$. Suppose that $u$ is a
solution to \eqref{eq-wave}.  Then
\begin{equation}\label{gradient}
|\nabla u|^2(x,y)  \le 2F\bigl(u(x,y)\bigr), \quad (x,y) \in \Er^n.
\end{equation}
\end{proposition}

This inequality can be proven  as in \cite{mod}  with minor
modifications.  The proof is omitted here. The reader is referred to
\cite{gui-hd} for a complete proof.

Now suppose $u$ is a solution to \eqref{eq-wave} and
\eqref{monotone} with limits $u^+=g_{\alpha}, \, u^- \equiv -1$.
Then, as in \eqref{u_y3} we have
\begin{equation}\label{uy1}
\begin{split}
h'(x)=&\int_{\Er} ( u_{xx}u_y + u_x u_{xy}) dy \\
=& -c\int_{\Er} u_y^2dy + F(g_{\alpha})+\frac{1}{2} |g_{\alpha}'|^2,
\quad \forall x \in \Er.
\end{split}
\end{equation}
Hence, in view of  the fact $F(g_{\alpha}(x)) \ge F(\alpha), \forall
x \in \Er$, we obtain
\begin{equation}\label{uy2}
c\int_{a}^{b} \int_{\Er} u_y^2dydx \le  h(a)-h(b) + F(\alpha) (b-a).
\end{equation}

On the other hand, by \eqref{gradient} there holds
\begin{equation}
 \int_{\Er} u_y^2dy \le \int_{\Er} u_y \sqrt{2F(u)} dy \le
 G(\beta)
 \end{equation}
 where $ \beta:=\inf_{x\in \Er} g_{\alpha}(x) \in (-1,0)$ with $F(\beta)=F(\alpha)$,  and
 $$G(s):=\int_{-1}^{s} \sqrt{2F(t)}dt >0, \quad \forall s \in
 (-1,1].
 $$
 Hence,
\begin{equation*}
 c>\frac{ F(\alpha)}{ G(\beta)}>0
 \end{equation*}
 and
 \begin{equation}
 c^*_{L}\ge \frac{ F(\alpha(L))}{ G(\beta(L))}>0.
 \end{equation}

Similarly,  if  $u$ is a solution to \eqref{eq-wave} and
\eqref{monotone} with limits $u^+=g_{*}, u^- \equiv -1$, as in
\eqref{uy1} we have
\begin{equation}\label{uy3}
\begin{split}
h'(x)=&\int_{\Er} ( u_{xx}u_y + u_x u_{xy}) dy \\
=& -c\int_{\Er} u_y^2dy + F(g_{*})+\frac{1}{2} |g_{*}'|^2, \quad
\forall x \in \Er.
\end{split}
\end{equation}
Hence, in view of $F(g_{*}(x)) \ge F(1), \forall x \in \Er$, we
obtain
\begin{equation}\label{uy4}
c\int_{a}^{b} \int_{\Er} u_y^2dydx \ge  h(a)-h(b) + F(1) (b-a).
\end{equation}

On the other hand,  by \eqref{gradient} there holds
\begin{equation}
 \lim_{x \to \infty} \int_{\Er} u_y^2dy \le \lim_{x \to \infty} \int_{\Er} u_y \sqrt{2F(u)} dy \le
 G(1)=\be.
 \end{equation}
 Hence,
 \begin{equation}
 c \ge \frac{ F(1)}{ \be}>0
 \end{equation}
and
 \begin{equation}
 c^*\ge  \frac{ F(1)}{\be }>0.
 \end{equation}

  The theorem is proven.

\end{proof}

The lower estimates of $c^*_{L}$ and $c_*$  above are obviously not
 optimal.  Finally, we would like to ask the following  questions.

\begin{question*}\nonumber Regarding Case (7), is it true that $c_L=c^*_L$?
When $c \ge c_L$, are all solutions to \eqref{eq-wave} and
\eqref{monotone} with limits $u^+=g_{\alpha(L)}, u^- \equiv -1$
periodic?   Regarding Case (9) for unbalanced $F$,  is it true that
$c_{\infty}=c^*$? When $c \ge c_{\infty}$, are all solutions to
\eqref{eq-wave} and \eqref{monotone} with limits $u^+=g_{*}, u^-
\equiv -1$ even in $x$ after a proper translation in $x$?
\end{question*}\nonumber

{\bf Acknowledgement} This  research  is partially supported by
National Science Foundation Grant DMS 0500871.

\end{document}